\documentclass[12pt]{article}
\usepackage[utf8]{inputenc}
\usepackage{amsmath,amsfonts,amssymb,amsthm,bm,color,mathrsfs,tikz}
\usepackage[colorinlistoftodos]{todonotes}
\usetikzlibrary{matrix,arrows,decorations.pathmorphing,decorations.markings}

\tikzset{negated/.style={
        decoration={markings,
            mark= at position 0.5 with {
                \node[transform shape] (tempnode) {$\backslash$};
            }
        },
        postaction={decorate}
    }
}

\newcommand{\A}{\boldsymbol{A}}
\newcommand{\B}{\boldsymbol{B}}
\newcommand{\C}{\boldsymbol{C}}
\newcommand{\D}{\boldsymbol{D}}
\newcommand{\RR}{\mathbb{R}}
\newcommand{\NN}{\mathbb{N}}
\newcommand{\ZZ}{\mathbb{Z}}

\newcommand{\cb}{\mathbf{c}}
\newcommand{\ab}{\mathbf{a}}

\newcommand{\vb}{\mathbf{v}}

\newtheorem{thm}{Theorem}
\newtheorem{lem}[thm]{Lemma}
\newtheorem{cor}[thm]{Corollary}

\newtheorem{Def}[thm]{Definition}
\newtheorem{ex}[thm]{Example}

\date{}
\title{A note on spectral properties of Hermite subdivision operators}
\author{
{Caroline Moosm\"uller}\thanks{Department of Chemical and Biomolecular
  Engineering, Johns Hopkins University, 3400 North Charles Street,
  Baltimore, MD 21218, USA. \texttt{cmoosmueller@jhu.edu}}
  }
\begin{document}

\maketitle

\begin{abstract}
In this paper we study the connection between the spectral condition of an Hermite subdivision operator and polynomial
reproduction properties of the associated subdivision scheme. While it is known that in general the spectral condition
does not imply the reproduction of polynomials, we here prove that a special spectral condition (defined by shifted monomials)
is actually equivalent to the reproduction of polynomials. We further put into evidence that the sum rule of order $\ell>d$ associated
with an Hermite subdivision operator of order $d$ does not imply that the spectral condition of order $\ell$ is satisfied, while it is
known that these two concepts are equivalent in the case $\ell=d$.

\end{abstract}
\section{Introduction}
Hermite subdivision schemes are iterative refinement algorithms, which, applied to discrete vector-valued data, aim at
producing smooth functions and their derivatives in the limit. They are defined by iteratively applying a subdivision operator
(and a normalization matrix due to the interpretation of the data as function values and consecutive derivatives)
to discrete vector-valued data \cite{dubuc06,dubuc05,dyn95,dyn99,merrien92}.
Therefore, properties such as convergence,
regularity of the limit curve, approximation order, etc.\ are strongly connected to properties of
the subdivision operator \cite{conti16,conti17,dubuc09,merrien12,moosmueller18b,yu05}.

Hermite subdivision schemes find applications in geometric modeling (if derivatives are of interest) \cite{dyn02,xue05,xue06},
in approximation theory (linear and manifold-valued) \cite{han06,merrien92,moosmueller16,moosmueller17}, and
they can be used for the construction of multiwavelets \cite{cotronei18,cotronei17,jia06} and the analysis of biomedical images \cite{conti15,uhlmann14}.

When studying Hermite subdivision operators it is commonly assumed that 
the spectral condition \cite{dubuc09,merrien12} is satisfied. This condition requires certain polynomial eigenvectors of the operator.
While the spectral condition is very useful for factorizing the subdivision operator (and thus for obtaining convergence results)
\cite{merrien12}, it has recently been proved that it is not necessary for convergence of the associated scheme \cite{merrien18}.
The paper \cite{merrien18} also shows how to relax the spectral condition, but still retain factorization and convergence results.
Furthermore, we recently showed that in general the spectral condition of order $d$ does not imply the reproduction of polynomials up to degree $d$
of the associated scheme \cite{moosmueller18b}. These rather surprising results are the motivation for this paper.

We aim at shedding some light on the relationship between the spectral condition and polynomial reproduction \cite{conti14,dyn99,jeong17},
polynomial generation \cite{dubuc09}, and the sum rules \cite{han03b,han05}.  



Reproduction of polynomials of degree $\ell$ is a desired property of a convergent (Hermite) scheme,
as it leads to approximation order $\ell+1$,
see \cite{conti11,dyn06,dyn02,jeong17,yu01}.
While it is known that for interpolatory Hermite schemes the reproduction of polynomials is necessary for convergence \cite{dyn99,yu05},
this is in general not true for noninterpolatory schemes.

The polynomial reproduction order of an Hermite scheme can be checked, for example, with the algebraic
conditions for the mask provided in \cite{conti18}.
We prove in the first part of this paper that a special spectral condition (namely, the spectral condition with shifted monomials
of the form $\tfrac{(x+\tau)^k}{k!}$) is equivalent to the reproduction of polynomials. This provides an easy way of checking the polynomial
reproduction order (and thus the approximation order) of an Hermite scheme.

In the second part of this paper we study the connection between the spectral condition and the sum rules of \cite{han03b,han05}.
In \cite{dubuc09} it is proved that these concepts are equivalent in the \emph{minimal} regime, meaning that both the spectral condition
and the sum rules are satisfied up to order $d$, where $d+1$ is the dimension of the mask coefficients.
We consider two examples with mask consisting of $(2\times 2)$ matrices from \cite{han05} which satisfy the sum rules of order $7$ and
prove that the spectral condition is only satisfied up to order $2$. This implies that
in general the sum rules of order $\ell > d$ do not imply the spectral condition of order $\ell$.
\section{Preliminaries}
Throughout this text let $d \geq 1$. 
We study the space of $\RR^{d+1}$-valued sequences $\cb=\left(\cb_j: j \in \ZZ\right)$, which is denoted by $\ell(\ZZ)^{d+1}$
and the space $\ell(\ZZ)^{d+1}_{\infty}$ of bounded sequences with respect to the norm
\begin{equation*}
\|\cb\|_{\infty}:=\sup_{j \in \ZZ}|\cb_j|_{\infty},
\end{equation*}
where $|\cdot|_{\infty}$ is the infinity norm on $\RR^{d+1}$.
We also consider spaces of matrix-valued sequences $\A=\left(\A_j: j \in \ZZ \right)$, denoted by $\ell(\ZZ)^{(d+1)\times (d+1)}$
and the space $\ell(\ZZ)^{(d+1)\times (d+1)}_{\infty}$ of bounded sequences with respect to the norm:
\begin{equation*}
\|\A\|_{\infty}:=\sup_{j \in \ZZ}|\A_j|_{\infty},
\end{equation*}
where $|\cdot|_{\infty}$ is the operator norm for matrices in $\RR^{(d+1)\times (d+1)}$ induced by the infinity norm on $\RR^{d+1}$.
We further define the spaces $\ell(\ZZ)^{d+1}_{0}$ and $\ell(\ZZ)^{(d+1)\times (d+1)}_{0}$, which consist of vector and matrix sequences,
respectively, of finite support.
\begin{Def}[Subdivision operator]\label{def:subd_op}
A \emph{subdivision operator of order $d$} with \emph{mask} $\A \in \ell(\ZZ)^{(d+1)\times (d+1)}_{0}$ is the map $S_{\A}:\ell(\ZZ)^{d+1} \to \ell(\ZZ)^{d+1}$ defined by
\begin{equation}\label{eq:sdo}
\left(S_{\A}\cb\right)_j=\sum_{k \in \ZZ}\A_{j-2k}\cb_j, \quad \cb \in \ell(\ZZ)^{d+1},\, j \in \ZZ.
\end{equation}
\end{Def}
\begin{Def}[Hermite subdivision scheme]\label{def:hermite}
Let $S_{\A}$ be a subdivision operator of order $d$. An \emph{Hermite subdivision scheme with operator $S_{\A}$} is the iterative procedure of constructing vector-valued sequences by
\begin{equation}\label{eq:Hermite_sds}
\D^{n+1}\cb^{[n+1]}=S_{\A}\D^{n}\cb^{[n]},\quad n\in \NN,
\end{equation}
from initial data $\cb^{[0]} \in \ell(\ZZ)^{d+1}$. Here $\D$ denotes the diagonal matrix $\D=\operatorname{diag}\left(1, 2^{-1},\ldots,2^{-d} \right)$.
\end{Def}
Note that (\ref{eq:Hermite_sds}) is equivalent to
\begin{equation}\label{eq:rewrite}
\cb^{[n]}=\D^{-n}S_{\A}^n\cb^{[0]},\quad n\in \NN.
\end{equation}

An Hermite subdivision scheme is called \emph{interpolatory} if $\cb^{[n+1]}_{2j}=\cb^{[n]}_{j},\, n\in \NN,\, j\in \ZZ$.
This is equivalent to the following conditions on the mask $\A$: $\A_0=\D$ and $\A_{2j}=0,\, j\in \ZZ \backslash \{0\}$.

We study parametrizations of Hermite subdivision schemes \cite{conti18,conti14,jeong17}, which are 
characterized by a parameter $\tau \in\RR$. We say that an Hermite subdivision scheme is parametrized by a parameter $\tau \in \RR$
if we consider the vector $\cb^{[n]}_j$ to be attached to the value $\tfrac{j+\tau}{2^n}$, $j\in \ZZ,\, n\in \NN$.
The choice $\tau=0$ is called the \emph{primal parametrization} and $\tau=-1/2$ is called the \emph{dual parametrization}. 

\begin{Def}[Convergence of Hermite subdivision schemes]\label{def:convergent_hermite}
An Hermite subdivision scheme with subdivision operator $S_{\A}$ of order $d$ is $C^{\ell}$-convergent, $\ell \geq d$, with parametrization $\tau$, if for every input data $\cb^{[0]}\in \ell(\ZZ)^{d+1}$, there exists a uniformly continuous function $\Phi=[\varphi_j]_{j=0}^d: \RR \to \RR^{d+1}$ such that for every compact set $K\subset \RR$ the sequence $\cb^{[n]}$ defined by (\ref{eq:Hermite_sds}) satisfies
\begin{equation}\label{eq:Hermite_convergence}
\lim_{n \to \infty}\sup_{j\in \ZZ\cap K}|\cb^{[n]}_j-\Phi\left(2^{-n}(j+\tau)\right)|_{\infty}=0
\end{equation}
and $\varphi_0 \in C^d(\RR)$ with $\tfrac{d^j\varphi_0}{dx^j}=\varphi_{j}, j=1,\ldots,d$.
Furthermore, we request that there exists at least one $\cb^{[0]} \in \ell(\ZZ)^{d+1}$ such that $\Phi \neq 0$.
\end{Def}
Denote by $\Pi_k$ the polynomials of degree $\leq k$ with real coefficients. Let $p \in \Pi_k$ and define $\vb_p \in \ell(\ZZ)^{d+1}$ by
\begin{equation}
(\vb_p)_j=[p(j),p'(j),\ldots,p^{(d)}(j)]^T,\quad j\in \ZZ.
\end{equation}
\begin{Def}[Spectral condition]\label{def:spectral}
A subdivision operator $S_{\A}$ of order $d$ satisfies the \emph{spectral condition of order $\ell$}, $\ell \geq d$, if there exist polynomials
$p_k \in \Pi_k$, $k=0,\ldots, \ell$, where $p_k$ has leading coefficient $1/k!$, such that
\begin{equation}\label{eq:spectral}
 S_{\A}\vb_{p_k}
= 2^{-k}\vb_{p_k}.
\end{equation}
A subdivision operator of order $d$ satisfying the spectral condition of order $\ell$ is called a \emph{subdivision operator of spectral order $\ell$}. The polynomials $p_k, k=0,\ldots, \ell$, are named \emph{spectral polynomials of $S_{\A}$}.
A subdivision operator of order $d$ is of \emph{minimal spectral order} if it is of spectral order $d$, but not of spectral order $d+1$.
\end{Def}
The spectral condition was introduced by \cite{dubuc08,dubuc09}.
In \cite{dubuc08} it is also proved that the spectral polynomials are unique up to a multiplicative factor.

The spectral condition is an important property for the factorizability of an Hermite subdivision scheme  \cite{merrien12,moosmueller18b}. Nevertheless, it is known that it is not necessary for the convergence of a scheme \cite{merrien18}.

\begin{Def}[Reproduction of functions]\label{def:poly_reprod}
Let $S_{\A}$ be a subdivision operator of order $d$. Then $S_{\A}$ is said to
\emph{reproduce a function $f\in C^{d}(\RR)$ with respect to $\tau$} if for initial data $\cb^{[0]}_j=[f(j+\tau),f'(j+\tau),\ldots,f^{(d)}(j+\tau)]^T$, the iterated sequence $\cb^{[n]}$ defined by (\ref{eq:Hermite_sds}) satisfies
$\cb^{[n]}_j=[f(2^{-n}(j+\tau)),f'(2^{-n}(j+\tau)),\ldots,f^{(d)}(2^{-n}(j+\tau))]^T$, $j\in \ZZ, n\geq 1$.
\end{Def}
\begin{Def}[Generation of functions]\label{def:poly_generation}
Let $S_{\A}$ be a subdivision operator of order $d$. It is said to \emph{generate a function $f\in C^{d}(\RR)$ with respect to $\tau$}
if the associated Hermite scheme is $C^{d}$-convergent and for some initial data $\cb^{[0]}$, the iterated sequence $\cb^{[n]}$ defined by (\ref{eq:Hermite_sds}) converges to $[f,f',\ldots,f^{(d)}]^T$ with parametrization $\tau$.
\end{Def}
\begin{Def}
Let $S_{\A}$ be a subdivision operator of order $d$. Let $\ell\geq d$ and let $\tau \in \RR$.
We say that $S_{\A}$ reproduces (generates) $\Pi_{\ell}$ w.r.t.\ $\tau$
if it reproduces (generates) all polynomials in $\Pi_{\ell}$ w.r.t.\ $\tau$.
\end{Def}
Note that in order to show reproduction (generation) of $\Pi_{\ell}$, due to linearity, it is enough to prove reproduction (generation)
of a basis of $\Pi_{\ell}$.

Polynomial reproduction and generation of Hermite subdivision schemes are studied in e.g.\ \cite{conti16,conti14,dubuc09,dyn99,jeong17}. Now we introduce a special sum rule from \cite{han03b,han05}:
\begin{Def}[Special sum rule]\label{def:sum_rule}
Let $S_{\A}$ be a subdivision operator of order $d$. It is said to satisfy a \emph{special sum rule of order $\ell \geq d$} if there exists $y\in \ell(\ZZ)^{d+1}$ such that
\begin{equation}
\frac{d^j}{dx^j}\left(\hat{y}(2\cdot)\hat{\A}(\cdot)\right)(0)=\frac{d^j}{dx^j}\hat{y}\,(0), \,
\frac{d^j}{dx^j}\left(\hat{y}(2\cdot)\hat{\A}(\cdot)\right)(\pi)=0,
\end{equation}
for $j=0,\ldots,\ell$ and
\begin{equation}
\frac{(-1)^j}{j!}\frac{d^j}{dx^j}\hat{y}\,(0)=e_{j}
\end{equation}
for $j=0,\ldots,d$. Here $e_{j}$ is the $(j+1)$-st unit vector in $\RR^{d+1}$, $j=0,\ldots,d$.

$S_{\A}$ is said to satisfy a \emph{minimal special sum rule} if it satisfies a special sum rule of order $d$, but not of order $d+1$.
\end{Def}
\section{Spectral condition and polynomial reproduction}\label{sec:polynomial}
In this section we study the relationship between the spectral condition (Definition \ref{def:spectral}), polynomial reproduction (Definition \ref{def:poly_reprod}) and polynomial generation (Definition \ref{def:poly_generation}).
The main result (Theorem \ref{thm:main}) shows that the reproduction of polynomials is equivalent to a special spectral condition. It also shows that we actually have an equivalence in \cite[Proposition 1]{conti14}.
We summarize the findings of this section in Figure \ref{fig:general_case}.
\begin{figure}
\centering
\begin{tikzpicture}
  \matrix (m) [matrix of math nodes,row sep=3em,column sep=4em,minimum width=2em]
  {
    \text{$S_{\A}$ reproduces $\Pi_{\ell}$ w.r.t.\ $\tau$} \\ 
    \text{ $S_{\A}$ satisfies the spectral condition with polynomials $(x+\tau)^k/k!, \, k=0,\ldots, \ell$.} \\
       \text{$S_{\A}$ generates $\Pi_{\ell}$ w.r.t.\ $\tau$}  \\};
  \path[-stealth]
    (m-1-1) edge [<->,double] (m-2-1)
    (m-2-1) edge [->,double] (m-3-1);
\end{tikzpicture}
\caption{This is a summary of the results presented in Theorem \ref{thm:main} and Lemma \ref{lem:generation}.
Here $S_{\A}$ is an Hermite subdivision operator of order $d$, $\ell\geq d$, and $\tau \in \RR$ is 
the parameter with respect to which we parametrize the
Hermite scheme associated to $S_{\A}$.}
\label{fig:general_case}
\end{figure}
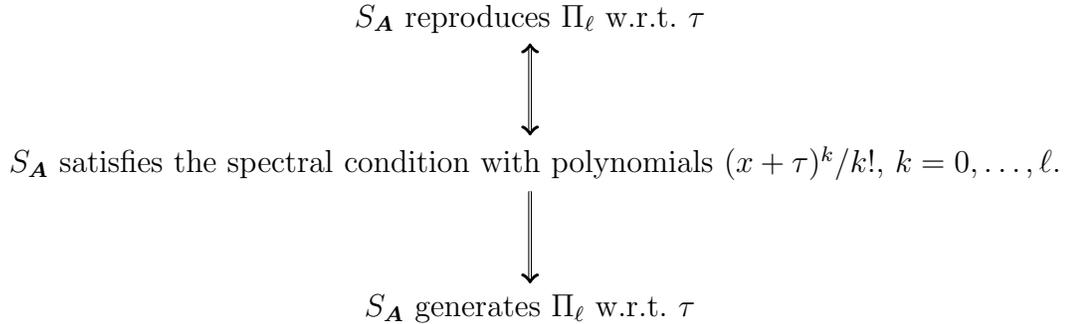

\begin{thm}\label{thm:main}
Let $S_{\A}$ be a subdivision operator of order $d$. Let $\ell \geq d$ and let $\tau \in \RR$.
Then we have: $S_{\A}$ satisfies the spectral condition with spectral polynomials $(x+\tau)^k/k!, k=0,\ldots,\ell$,
if and only if it reproduces $\Pi_{\ell}$ with respect to the parameter $\tau$.
\end{thm}
\begin{proof}
Let $q_k(x)=x^k/k!$ and $p_{\tau,k}(x)=q_k(x+\tau), k=0,\ldots, \ell$.

$(\Rightarrow)$ The set ${\mathcal{B}}=\{ q_k, k=0,\ldots,\ell \}$ 
is a basis of $\Pi_{\ell}$, hence if $\mathcal{B}$ is reproduced by the Hermite scheme, then due to linearity, all of $\Pi_{\ell}$ is reproduced. 

By assumption, the spectral condition with spectral polynomials $p_{\tau,k},\, k=0,\ldots,\ell$,
is satisfied and we now prove that ${\cal B}$
is reproduced with parametrization $\tau$. Consider initial data 
\begin{equation}\label{eq:inital}
(\cb^{[0]}_k)_j=[q_k(j+\tau),\ldots,q_k^{(d)}(j+\tau)]^T=(\vb_{p_{\tau,k}})_j, \quad k=0,\ldots, \ell, j\in \ZZ.
\end{equation}
Since
\begin{equation*}
q_k^{(m)}(x)= \begin{cases}
				\frac{x^{k-m}}{(k-m)!} & m \leq k \\
                0 & m>k
                \end{cases},
\end{equation*}
for $n\in \NN$ and $m \leq k$ we obtain
\begin{equation}\label{eq:deriv}
q_k^{(m)}(2^{-n}(j+\tau))=2^{-n(k-m)}\frac{(j+\tau)^{k-m}}{(k-m)!}=2^{-n(k-m)}q_k^{(m)}(j+\tau).
\end{equation}
Note that (\ref{eq:deriv}) is also satisfied if $m>k$, because in this case both sides of the equation equal $0$.
The spectral condition with spectral polynomials $p_{\tau,k}, k=0,\ldots,\ell,$ gives
\begin{equation*}
S_{\A}\vb_{p_{\tau,k}}=2^{-k}\vb_{p_{\tau,k}}, \quad k=0,\ldots,\ell,
\end{equation*}
and by iteration
\begin{equation*}
S_{\A}^n\vb_{p_{\tau,k}}=2^{-kn}\vb_{p_{\tau,k}}, \quad n\in \NN,\, k=0,\ldots,\ell.
\end{equation*}
Together with (\ref{eq:rewrite}) this implies 
\begin{align*}
\cb^{[n]}_k=\D^{-n}S_{\A}^n\cb^{[0]}_k
=\D^{-n}S_{\A}^n\vb_{p_{\tau,k}}=\D^{-n}2^{-kn}\vb_{p_{\tau,k}}
=\D^{-n}2^{-kn}\cb^{[0]}_k,
\end{align*}
for $n\in \NN$ and $k=0,\ldots,\ell$.
Thus using the relation (\ref{eq:deriv}), for $m=0,\ldots,d$, the $m$-th component of $\cb^{[n]}$ is given by
\begin{align*}
(\cb^{[n]}_k)^{m}_j=2^{-n(k-m)}(\cb^{[0]}_k)^m_j=2^{-n(k-m)}q_k^{(m)}(j+\tau)=q_k^{(m)}(2^{-n}(j+\tau)),
\end{align*}
for $n\in \NN,\, k=0,\ldots, \ell$. This concludes the first part of the proof.

$(\Leftarrow)$ For $\tau=0$, this is proved in \cite{conti14}.
We provide the proof for general $\tau$. We assume that $\Pi_{\ell}$ is reproduced with parametrization $\tau$.
Therefore, in particular, $q_{k}, k=0,\ldots,\ell,$ are reproduced. As in (\ref{eq:inital}), let $\cb^{[0]}_k=\vb_{p_{\tau,k}}$.
Then polynomial reproduction implies
\begin{equation*}
(\cb^{[n]}_k)_j=
[q_k(2^{-n}(j+\tau)),q_k'(2^{-n}(j+\tau)),\ldots,q_k^{(d)}(2^{-n}(j+\tau))]^T,
\end{equation*}
for $n\in \NN,\, j\in \ZZ$ and $k=0,\ldots,\ell$.
In particular, using this formula for $n=1$ and together with (\ref{eq:deriv}) we get
\begin{align*}
(\cb^{[1]}_k)_j&=
[q_k(2^{-1}(j+\tau)),q_k'(2^{-1}(j+\tau)),\ldots,q_k^{(d)}(2^{-1}(j+\tau))]^T\\
&=[2^{-k}q_k(j+\tau),2^{-(k-1)}q_k'(j+\tau),\ldots,2^{-(k-d)}\,q_k^{(d)}(j+\tau)]^T\\
&=\D^{-1}2^{-k}[q_k(j+\tau),q_k'(j+\tau),\ldots,q_k^{(d)}(j+\tau)]^T\\
&=\D^{-1}2^{-k}(\vb_{p_{\tau,k}})_j.
\end{align*}
This, together with (\ref{eq:rewrite}), implies
\begin{equation}
S_{\A}\vb_{p_{\tau,k}}=S_{\A}\cb^{[0]}_k=\D\cb^{[1]}_k=2^{-k}\vb_{p_{\tau,k}},
\end{equation}
which concludes the proof.
\end{proof}
Theorem \ref{thm:main} has some interesting implications:
\begin{enumerate}
\item \label{en:1} Theorem \ref{thm:main} shows that polynomial reproduction is equivalent to the spectral
condition with very special spectral polynomials. Due to this reason, it can happen that the spectral condition
of order $\ell$ is satisfied, but polynomials up to degree $\ell$ are not reproduced: In \cite{moosmueller18b}
we study the Hermite scheme $H_1$ (with $\theta=1/32$) by \cite{jeong17} and show that it satisfies the spectral condition of order $4$ with spectral
polynomials $1,x,\tfrac{x^2}{2!},\tfrac{x^3}{3!},\tfrac{x^4}{4!}+\tfrac{1}{360}$, but it does not satisfy the spectral
condition with $4$-th spectral polynomial given by $\tfrac{x^4}{4!}$. Therefore, it reproduces polynomials up to degree $3$ with primal
parametrization, but it does not reproduce all polynomials of degree $4$.

This is quite an interesting example. Since we have the spectral condition up to order $4$,
factorizations of the subdivision operator up to order $4$ are possible, see \cite{jeong17,moosmueller18b}.
It can even be proved that this scheme produces $C^4$ limits \cite{jeong17,moosmueller18b}, even though it only reproduces polynomials up to degree $3$.
This implies that polynomial reproduction is not necessary for the convergence and regularity of Hermite schemes
(which has been known before, see e.g.\ the examples in \cite{jeong17}),
even though this is known to be the case for interpolatory schemes \cite{dyn99}.
\item If the spectral condition of order $1$ is satisfied with first spectral polynomial given by $p_{\tau,1}(x)=x+\tau$,
then $\tau$ determines the parametrization with respect to which polynomials up to degree $1$ are reproduced.
In this sense, in the case $d=\ell=1$, \emph{any} spectral condition is equivalent to the reproduction of polynomials up to degree $1$.
\item Theorem \ref{thm:main} also explains why the algebraic conditions equivalent to the reproduction of polynomials up to order $1$ derived in \cite{conti18} are the same as the algebraic conditions equivalent to the spectral condition of order $1$ derived in \cite{moosmueller18}.
\item In the case $d=1$ and $d=2$ and $\ell\geq d$ the algebraic conditions equivalent to the reproduction of polynomials up to order $\ell$ of \cite{conti18} now also describe the spectral condition of order $\ell$ with spectral polynomials $(x+\tau)^k/k!,\, k=0,\ldots,\ell$.
\end{enumerate}
The following Lemma is an extension of \cite[Theorem 8]{dubuc09}. By adapting their proof we obtain
\begin{lem}\label{lem:generation}
Let $S_{\A}$ be a subdivision operator of order $d$ and let the Hermite scheme associated with $S_{\A}$
be $C^d$-convergent with parametrization $\tau$.
If $S_{\A}$ satisfies the spectral condition of order $\ell$, $\ell\geq d$, then $S_{\A}$ generates $\Pi_{\ell}$ w.r.t.\ $\tau$.
\end{lem}
\begin{proof}
For $k=0,\ldots,\ell$, denote by $p_{\tau,k}(x)=\sum_{s=0}^{k}c_{ks}(x+\tau)^s$, $c_{ks}\in \RR$, the spectral polynomials of $S_{\A}$.
Note that we expand $p_{\tau,k}$ in a different basis compared to Definition \ref{def:spectral}. However, it is easy to see that also in this
basis $c_{kk}=1/k!$. 
Consider input data
\begin{equation*}
(\cb^{[0]}_k)_j=
[p_{\tau,k}(j),p_{\tau,k}'(j),\ldots,p_{\tau,k}^{(d)}(j)]^T.
\end{equation*}
Since for $n\in \NN$ and $k=0,\ldots,\ell$ iterating the spectral condition gives $S_{\A}^{n}\vb_{p_{\tau,k}}=2^{-kn}\vb_{p_{\tau,k}}$,
with (\ref{eq:rewrite}) we obtain
\begin{align*}
\cb^{[n]}_k=\D^{-n}S_{\A}^{n}\cb^{[0]}_k=\D^{-n}2^{-kn}\cb^{[0]}_k.
\end{align*}
We show that the limit of $\cb^{[n]}_k$ is $q_k(x)=\tfrac{x^k}{k!}$:
Let $\Phi_k(x)=[q_k(x),\ldots, q_k^{(d)}(x)]^T, x\in \RR$, and let $M_k=\min\{d,k\}$. Then with (\ref{eq:deriv}) we obtain
\begin{align*}
|(\cb^{[n]}_k)_j-&\Phi_k(2^{-n}(j+\tau))|_{\infty}
=\max_{m=0,\ldots,d}|(\cb^{[n]}_k)_j^m-q_k^{(m)}(2^{-n}(j+\tau))|\\
&=\max_{m=0,\ldots,M_k}|2^{-n(k-m)}p_{\tau,k}^{(m)}(j)-2^{-n(k-m)}q_k^{(m)}(j+\tau)|.
\end{align*}
Now for $m=0,\ldots, M_k$, we have $q^{(m)}_k(j+\tau)=\tfrac{(j+\tau)^{k-m}}{(k-m)!}$. Furthermore, 
\begin{equation*}
p_{\tau,k}^{(m)}(j)=\sum_{s=m}^k c_{ks}s\cdots (s-m+1)(j+\tau)^{s-m}
=\sum_{s=m}^k c_{ks}s!\frac{(j+\tau)^{s-m}}{(s-m)!}.
\end{equation*}
This implies
\begin{equation*}
|2^{-n(k-m)}(p_{\tau,k}^{(m)}(j)-q^{(m)}_k(j+\tau))|\leq 2^{-n(k-m)}\sum_{s=m}^{k-1} c_{ks}s!\frac{|j+\tau|^{s-m}}{(s-m)!}.
\end{equation*}
For every $m$, the sum in the above term is bounded if $j\in \ZZ\cap K$ for a compact set $K\subset \RR$.
Thus for every compact $K\subset \RR$ we obtain
\begin{equation*}
\lim_{n\to \infty}\sup_{j \in\ZZ \cap K}|(\cb^{[n]}_k)_j-\Phi_k(2^{-n}(j+\tau))|_{\infty}=0.
\end{equation*}
Therefore the basis $\mathcal{B}=\{q_k,k=0,\ldots,\ell\}$ is generated by the scheme, which implies that all of $\Pi_{\ell}$ is generated.
\end{proof}
\subsection{Application to the de Rham transform}
We study the de Rham transform of an Hermite subdivision scheme, which has been introduced in \cite{dubuc08}. The de Rham transform is an interesting construction, as in some examples it increases the regularity of a scheme, while retaining a reasonable support size \cite{conti14}. In this section we prove that reproduction of $\Pi_{\ell}$ carries over to the de Rham transform.

Following \cite{conti14,dubuc08}, the de Rham transform of an Hermite scheme with mask $\A$ is the Hermite subdivision scheme with mask
\begin{align}\label{eq:derham_mask}
\overline{\A}_j:=\D^{-1}\sum_{m\in \ZZ}\A_{2(j-m)+1}\A_m, \quad j\in\ZZ.
\end{align}
Using notation introduced in e.g.\ \cite{sauer02}, we can rewrite (\ref{eq:derham_mask}):
\begin{align*}
\overline{\A}_j=\D^{-1}(\A\ast_2\A)_{2j+1}, \quad j\in\ZZ,
\end{align*}
where
\begin{equation*}
 (\B \ast_2 \C)_j:=S_{\B}\C_j=\sum_{m\in \ZZ}\B_{j-2m}\C_m, \quad j\in \ZZ,
\end{equation*}
with $\B,\C \in \ell(\ZZ)_0^{(d+1)\times (d+1)}$.
With the methods of \cite{dubuc08}, we now extend the results of \cite[Corollary 6]{dubuc08} and \cite[Corollary 1]{conti14} 
to general parametrizations $\tau$:
\begin{lem}\label{lem:deRham}
Let $S_{\A}$ be a subdivision operator of order $d$. Let $\ell\geq d$ and let $\tau \in \RR$.
If $S_{\A}$ satisfies the spectral condition with spectral polynomials $(x+\tau)^k/k!,\, k=0,\ldots, \ell$, then its de Rham transform $S_{\overline{\A}}$ satisfies the spectral condition with spectral polynomials $(x+2^{-1}(3\tau-1))^k/k!, \, k=0,\ldots, \ell$.
\end{lem}
\begin{proof}
In \cite[Theorem 5]{dubuc08} and \cite[Theorem 3]{conti14} it is proved that if $S_{\A}$ satisfies the spectral
condition with spectral polynomials $p_k, k=0,\ldots,\ell$, then $S_{\overline{\A}}$ satisfies the spectral condition
with spectral polynomials $\overline{p}_k$ given recursively by
\begin{align}\nonumber
& \overline{p}_0(x)=p_0(x)=1,\\ \label{al:rec}
& \overline{p}_k(x)=p_k(x)+\sum_{m=0}^{k-1}\mu_{k,m}\overline{p}_{m}(x), \quad k\geq 1,
\end{align}
with $\mu_{k,m}=-\lambda_{k,m}\frac{2^{m-k}}{2^k-2^m}$. Here the coefficients $\lambda_{k,m}$ are given by
\begin{equation*}
p_k(2x+1)=2^kp_k(x)+\sum_{m=0}^{k-1}\lambda_{k,m}\overline{p}_m(x).
\end{equation*}
With this at hand, we prove our Lemma recursively: For $k=0$ it is obviously true. Assume that it is true for $k-1$, we prove it for $k$.
We start with the coefficients $\lambda_{k,m}$:
\begin{align*}
&p_k(2x+1)=\frac{(2x+1+\tau)^k}{k!}=
2^k\frac{(x+\tau)^k}{k!}-2^k\frac{(x+\tau)^k}{k!}+2^k\frac{(x+2^{-1}+2^{-1}\tau)^k}{k!}\\
&=2^kp_k(x)-2^k\frac{(x+2^{-1}(3\tau-1)+2^{-1}(1-\tau))^k}{k!}+2^k\frac{(x+2^{-1}(3\tau-1)+1-\tau)^k}{k!}\\
&=2^kp_k(x)+\frac{2^k}{k!}
\sum_{m=0}^k {k\choose m}(x+2^{-1}(3\tau-1))^m\left((1-\tau)^{k-m}-(2^{-1}(1-\tau))^{k-m}\right)\\
&=2^kp_k(x)+2^k
\sum_{m=0}^{k-1} \frac{1}{(k-m)!}\frac{(x+2^{-1}(3\tau-1))^m}{m!}\left(1-2^{m-k}\right)(1-\tau)^{k-m}\\
&=2^kp_k(x)+2^k
\sum_{m=0}^{k-1} \frac{1}{(k-m)!}\overline{p}_m(x)\left(1-2^{m-k}\right)(1-\tau)^{k-m}\\
&=2^kp_k(x)+
\sum_{m=0}^{k-1} \lambda_{k,m}\overline{p}_m(x),
\end{align*}
with
\begin{equation*}
\lambda_{k,m}=\frac{2^k-2^m}{(k-m)!}(1-\tau)^{k-m}.
\end{equation*}
Therefore
\begin{equation*}
\mu_{k,m}=-\frac{2^{m-k}}{(k-m)!}(1-\tau)^{k-m}.
\end{equation*}

Now using the recursive relation (\ref{al:rec}) we obtain
\begin{align*}
\overline{p}_{k}(x)=&\,\frac{(x+\tau)^k}{k!}+\sum_{m=0}^{k-1}\mu_{k,m}\frac{(x+2^{-1}(3\tau-1))^m}{m!}\\
=&\,\frac{(x+2^{-1}(3\tau-1)+2^{-1}(1-\tau))^k}{k!}\\
&-\sum_{m=0}^{k-1}\frac{2^{m-k}}{(k-m)!}(1-\tau)^{k-m}\frac{(x+2^{-1}(3\tau-1))^m}{m!}\\
=&\,\frac{1}{k!}\sum_{m=0}^{k}{k \choose m}(2^{-1}(1-\tau))^{k-m}(x+2^{-1}(3\tau-1))^m\\
&-\frac{1}{k!}\sum_{m=0}^{k-1}2^{m-k}{k\choose m}(1-\tau)^{k-m}(x+2^{-1}(3\tau-1))^m\\
&=\frac{(x+2^{-1}(3\tau-1))^k}{k!},
\end{align*}
which concludes the proof.
\end{proof}
Combining Theorem \ref{thm:main} and Lemma \ref{lem:deRham} we obtain
\begin{cor}
Let $S_{\A}$ be a subdivision operator of order $d$. Let $\ell\geq d$ and let $\tau \in \RR$.
If $S_{\A}$ reproduces $\Pi_{\ell}$ w.r.t.\ $\tau$, then its de Rham transform $S_{\overline{\A}}$ reproduces $\Pi_{\ell}$
w.r.t.\ $2^{-1}(3\tau-1)$.
\end{cor}

\section{Spectral condition and the special sum rule}\label{sec:sum_rule}
In this section we study the connection between the spectral condition (Definition \ref{def:spectral}) and the special sum rule (Definition \ref{def:sum_rule}). We show that there is actually a difference between the minimal case and the general case.
The results of this section are summarized in Figures \ref{fig:sum_rule_minimal} and \ref{fig:sum_rule_general_case}. In these figures we also added that the spectral condition in general does not imply the reproduction of polynomials, a fact which was shown in \cite{moosmueller18b} and is here discussed in Section \ref{sec:polynomial}.
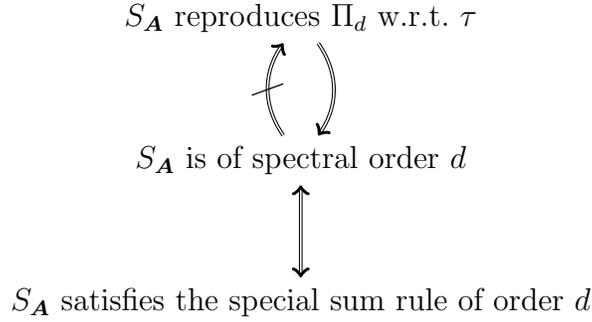
\begin{figure}
\centering
\begin{tikzpicture}
  \matrix (m) [matrix of math nodes,row sep=3em,column sep=4em,minimum width=2em]
  {
    \text{$S_{\A}$ reproduces $\Pi_{d}$ w.r.t.\ $\tau$} \\ 
    \text{$S_{\A}$ is of spectral order $d$} \\
       \text{$S_{\A}$ satisfies the special sum rule of order $d$}  \\};
  \path[-stealth]
    (m-2-1) edge [<-, double, bend right=35] (m-1-1)
    (m-2-1) edge [negated,->,double,bend left=35] (m-1-1)
    (m-2-1) edge [ <->,double] (m-3-1);
\end{tikzpicture}
\caption{This is a summary of results from \cite{conti14,dubuc09,moosmueller18b}, which we discuss in
Sections \ref{sec:polynomial} and \ref{sec:sum_rule}.
Here $S_{\A}$ is an Hermite subdivision operator of order $d$, and $\tau \in \RR$ is 
the parameter with respect to which we parametrize the
Hermite scheme associated to $S_{\A}$.
We consider properties of $S_{\A}$ of minimal order $\ell=d$.}
\label{fig:sum_rule_minimal}
\end{figure}

From \cite{dubuc09} we have the following result:
\begin{lem}\label{lem:minimal}
Let $S_{\A}$ be a subdivision operator of order $d$. Then $S_{\A}$ satisfies the minimal spectral condition if and only if it satisfies the minimal special sum rule.
\end{lem}
We now put into evidence that this result does not extend to the general case $\ell >d$.
We study two Hermite subdivision schemes from \cite{han05}. 
It is proved there that these scheme satisfy the special sum rule of order $7$.
We show that the spectral condition up to order $2$ is satisfied, but higher spectral conditions are not satisfied:
\begin{ex}
We consider the Hermite scheme with mask $\ab^1$ from \cite[Section 3.1]{han05}. It is supported in $[-2,2]\cap \ZZ$:
\begin{align*}
&\left[ \begin{array}{cc}
               1/128 & 7/256 \\
               0 & 1/16
              \end{array} 
        \right], \,
\left[ \begin{array}{cc}
          1/2 & -1/16 \\
          15/16 & -7/32
\end{array}
        \right], \,
\left[ \begin{array}{cc}
               63/64 & 0 \\
               0 & 3/8
              \end{array}
 \right],\\[0.2cm]
& \left[ \begin{array}{cc}
               1/2 & 1/16 \\
               -15/16 & -7/32
              \end{array}
 \right],
\left[ \begin{array}{cc}
               1/128 & -7/256 \\
               0 & 1/16
              \end{array}
 \right],
\end{align*}
and the Hermite scheme with mask $\ab^2$ from \cite[Section 3.1]{han05}, also supported in $[-2,2]\cap \ZZ$:
\begin{align*}
&\left[ \begin{array}{cc}
               7/96 & -25/1344 \\
               77/384 & -19/384
              \end{array} 
        \right], \,
\left[ \begin{array}{cc}
          1/2 & -5/56 \\
          7/12 & -1/24
\end{array}
        \right], \,
\left[ \begin{array}{cc}
               41/48 & 0 \\
               0 & 19/96
              \end{array}
 \right],\\[0.2cm]
& \left[ \begin{array}{cc}
               1/2 & 5/56 \\
               -7/12 & -1/24
              \end{array}
 \right],
\left[ \begin{array}{cc}
               7/96 & 25/1344 \\
               -77/384 & -19/384
              \end{array}
 \right].
\end{align*}
Note that we had to transpose the masks from \cite{han05}, due to the different notation of subdivision operator
used there.

It is shown in \cite{han05} that both schemes satisfy the special sum rule of order $7$. Furthermore,
it is easy to see that $S_{\ab^1}$ satisfies the spectral condition of order $2$ with spectral polynomials
$1,x,\tfrac{1}{2!}x^2-\tfrac{1}{12}$,
but it does not satisfy the spectral condition of order $3$.
Similarly, $S_{\ab^2}$ satisfies the spectral condition of order $2$ with spectral polynomials $1,x,\tfrac{1}{2!}x^2-\tfrac{1}{21}$, but it does not satisfy the spectral condition of order $3$. Note that we already knew from \cite{dubuc09} (summarized in our Lemma \ref{lem:minimal}) that these schemes satisfy the spectral condition of order $1$.

Theorem \ref{thm:main} further implies that $S_{\ab^1}$ and $S_{\ab^2}$
reproduce $\Pi_1$ with primal parametrization, but they do not reproduce $\Pi_2$ (which has already been noted in \cite{jeong17}, though without proof).

Therefore, similar to the scheme $H_1$ of \cite{jeong17}, these schemes satisfy spectral conditions of higher order than their polynomial reproduction order, see also the discussion in Section \ref{sec:polynomial} and \cite{moosmueller18b}.

With the factorization framework \cite{moosmueller18b}, regularity up to $C^2$ can be proved,
even though from \cite{han05} the scheme $S_{\ab^1}$ is $C^3$ and the scheme $S_{\ab^2}$ is $C^5$. These examples show that the spectral condition is not necessary for convergence, a fact which has also been noted recently in \cite{merrien18}.
\end{ex}
\begin{figure}
\centering
\begin{tikzpicture}
  \matrix (m) [matrix of math nodes,row sep=3em,column sep=4em,minimum width=2em]
  {
    \text{$S_{\A}$ reproduces $\Pi_{\ell}$ w.r.t.\ $\tau$} \\ \text{$S_{\A}$ is of spectral order $\ell$} \\
       \text{$S_{\A}$ satisfies a special sum rule of order $\ell$}   \\};
  \path[-stealth]
    (m-2-1) edge [<-, double, bend right=35] (m-1-1)
    (m-2-1) edge [negated,->,double,bend left=35] (m-1-1)
    (m-2-1) edge [negated, <-,double] (m-3-1);
\end{tikzpicture}
\caption{This is a summary of some of the results presented in Sections \ref{sec:polynomial} and \ref{sec:sum_rule}. 
Here $S_{\A}$ is an Hermite subdivision operator of order $d$, $\ell> d$, and $\tau \in \RR$ is 
the parameter with respect to which we parametrize the
Hermite scheme associated to $S_{\A}$.
}
\label{fig:sum_rule_general_case}
\end{figure}
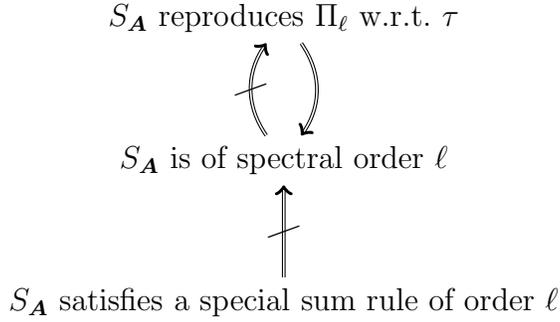
\section{Conclusion}
In this paper we study spectral properties of Hermite subdivision operators. Even though the spectral condition 
of an Hermite subdivision operator is not necessary for the convergence
of the associated scheme \cite{merrien18}, it plays an important role in the factorizability of the operator and
the regularity of the limit \cite{merrien12,moosmueller18b}. We prove that the reproduction of polynomials with respect to
a parameter $\tau$ is equivalent to the spectral condition with shifted monomials of the form $(x+\tau)^k/k!$, extending and 
generalizing a result of \cite{conti14}. We further extend a result of \cite{dubuc09} on polynomial generation, and apply our findings
to the de Rham transform \cite{conti14,dubuc08}.
In the last part of the paper we put into evidence that the special sum rule of order $\ell >d$,
where $d+1$ is the dimension of the mask coefficients, does not imply the spectral condition, even though these notions are known
to be equivalent if $\ell=d$ \cite{dubuc09}.

This paper aims at instigating research on (spectral) properties of Hermite subdivision operators and reproduction/generation properties of
the associated scheme, as there seem to be subtle differences between these notions. We would thus like to conclude this paper with three open questions:
\begin{enumerate}
 \item Does generation of polynomials imply the spectral condition?
 \item Does the spectral condition of order $\ell > d$ imply the special sum rule of order $\ell$?
 \item How is the special sum rule connected to polynomial reproduction in the case $\ell >d$?
\end{enumerate}

\section*{Acknowledgments}
The author thanks Svenja H\"uning for her valuable comments and suggestions and
the Department of Chemical and Biological Engineering, Princeton
University, for their hospitality.

\bibliographystyle{plain}

\end{document}